\documentclass[10pt,a4]{article}

\usepackage{latexsym}
\usepackage{amsmath}
\usepackage{amssymb}
\usepackage{amsthm}
\usepackage{amscd}
\usepackage{mathrsfs}
\usepackage[all]{xy}
\usepackage{graphicx}
\usepackage{bm}

\newtheorem{thm}{Theorem}[section]
\newtheorem{prop}[thm]{Proposition}
\newtheorem{defn}[thm]{Definition}
\newtheorem{lem}[thm]{Lemma}

\newtheorem{rem}[thm]{Remark}

\makeatletter

\@addtoreset{equation}{section}
\makeatother

\DeclareMathOperator{\Spec}{Spec}

\DeclareMathOperator{\Spf}{Spf} 
\DeclareMathOperator{\Sp}{Sp}

\DeclareMathOperator{\Ker}{Ker}
\DeclareMathOperator{\Coker}{Coker}

\begin{document}
\title
{Cohomology of rigid curves with semi-stable coverings}
\author{Naoki Imai and Takahiro Tsushima}
\date{}
\maketitle

\begin{center}
\textit{Dedicated to Professor Robert Coleman\\ 
with respect for his achievements}
\end{center}


\begin{abstract}
We construct a semi-stable formal model of 
a wide open rigid curve 
with a semi-stable covering, and 
study the $\ell$-adic cohomology of the rigid curve. 
We describe the $\ell$-adic cohomology of the rigid curve 
using the $\ell$-adic cohomology of the irreducible components of 
a semi-stable reduction, and 
homology and cohomology of some graphs. 
We also prove 
the functoriality of the description 
for a finite flat morphism 
that is compatible with semi-stable coverings of 
wide open rigid curves. 
\end{abstract}

\section*{Introduction}
Let $K$ be a complete discretely valued field 
with a non-trivial valuation. 
We assume that the residue field $k$ of $K$ 
is an algebraic extension of a finite field. 
We consider a wide open rigid curve 
$W$ over $K$ with a semi-stable covering. 
The notion of a semi-stable covering of 
a wide open rigid curve 
is due to Coleman and 
Coleman-McMurdy (cf. \cite{CoRec}, \cite{CMp^3}). 
A semi-stable covering of  
a wide open rigid curve is an analogue of 
a semi-stable model of an algebraic curve. 
In fact, Coleman-McMurdy constructed 
a semi-stable model from 
a semi-stable covering of a proper smooth curve 
(cf. \cite[Theorem 2.36]{CMp^3}). 
In this paper, we construct a semi-stable formal model 
of $W$ from a 
semi-stable covering of $W$. 

Let $\ell$ be a prime number different from 
the characteristic of $k$. 
The purpose of this paper is to 
study the $\ell$-adic cohomology of $W$. 
Let $\mathscr{W}$ be a semi-stable formal model of 
$W$ constructed from 
the semi-stable covering of $W$, and let 
$\Gamma_{\mathscr{W}}$ 
be the dual graph of the 
geometric closed fiber of $\mathscr{W}$. 
In this paper, we describe 
the $\ell$-adic cohomology of $W$ 
using 
the $\ell$-adic cohomology of the 
irreducible components of the 
geometric closed fiber of $\mathscr{W}$, 
and homology and cohomology 
of some variants of 
$\Gamma_{\mathscr{W}}$. 

We also study a relative situation. 
Let $W_1$ and $W_2$ be 
wide open rigid curves with semi-stable coverings. 
We consider a finite flat morphism $f \colon W_1 \to W_2$ 
that is compatible with the semi-stable coverings. 
We show that such a morphism extends to a morphism between 
their formal semi-stable models. 
The pushforward and the pullback on 
the $\ell$-adic cohomology by $f$ 
induce morphisms on 
the $\ell$-adic cohomology of the 
irreducible components of the 
geometric closed fibers, 
and homology and cohomology 
of the graphs. 
The induced morphism on 
the $\ell$-adic cohomology of the 
irreducible components of the 
geometric closed fibers is a usual one. 
We will describe the morphisms 
on homology and cohomology 
of the graphs 
using the induced morphisms on 
graphs. 

The connected components of Lubin-Tate spaces for $GL_2$ are 
examples of wide open rigid curves. 
The intention of our research 
is in the application to the 
study of group actions on 
the $\ell$-adic cohomology of 
Lubin-Tate spaces for $\textit{GL}_2$. 

\subsection*{Acknowledgment}
The authors thank Yoichi Mieda for reading a manuscript of this paper 
and giving a lot of suggestions. 
They thank also Yuichiro Hoshi and Seidai Yasuda for 
helpful conversations. 
They are grateful to a referee for suggestions for improvements. 

\subsection*{Notation}
Throughout this paper, we use the following notation. 
For a field $L$ with 
a non-trivial non-archimedean valuation, 
the ring of integers of $L$ is 
denoted by $\mathcal{O}_L$. 
For a field $F$, 
the algebraic closure of $F$ 
is denoted by $\overline{F}$, and 
the absolute Galois group of $F$ is denoted by 
$G_F$. 
For a vector space $V$ over a field $F$, 
the dual vector space of $V$ over $F$ is 
denoted by $V^*$. 
For a commutative ring $A$, 
a commutative $A$-algebra $B$ and 
a scheme $X$ over $A$, 
the base change of $X$ to $B$ is denoted by $X_B$. 
For an extension 
$L_2$ over $L_1$ of fields with 
nontrivial non-archimedean complete valuations 
and a rigid space $X$ over $L_1$, 
the base change of $X$ to $L_2$ is 
denoted by $X_{L_2}$. 

\section{Semi-stable covering}
In this section, 
we recall the notion of 
a semi-stable covering 
and some related results from 
\cite{CMp^3}. 

Let $K$ be a complete discretely valued field 
with a non-trivial valuation $v$. 
We assume that the residue field $k$ of $K$ 
is an algebraic extension of a finite field 
of characteristic $p$. 
We normalize the valuation so that 
$v(K^{\times} )\ = \mathbb{Z}$. 
We put $v(0) =+\infty$ and 
$|x|=p^{-v(x)}$ for $x \in K$. 
The maximal ideal of $\mathcal{O}_K$ 
is denoted 
by $\mathfrak{m}_K$. 
Let $\mathbf{C}$ be the completion of 
an algebraic closure of $K$. 
The absolute value $|\cdot |$ on $K$ 
naturally extends to $|\cdot |$ 
on $\mathbf{C}$. 

For $r \in |\mathbf{C}^{\times} |$, 
let $B_K [r]$ and $B_K (r)$ be the rigid spaces over $K$ 
whose $\mathbf{C}$-valued points are 
$\{ x \in \mathbf{C} \ | \ |x| \leq r \}$ and 
$\{ x \in \mathbf{C} \ | \ |x| < r \}$, 
which we call a closed disk and an open disk respectively. 
For $r,s \in |\mathbf{C}^{\times} |$ satisfying $r \leq s$, 
let $A_K [r,s]$ and $A_K (r,s)$ be the rigid spaces over $K$ 
whose $\mathbf{C}$-valued points are 
$\{ x \in \mathbf{C} \ | \ r \leq |x| \leq s \}$ and 
$\{ x \in \mathbf{C} \ | \ r < |x| < s \}$, 
which we call a closed annulus and an open annulus respectively. 
A closed annulus of the form $A_K [r,r]$ for 
$r \in |\mathbf{C}^{\times} |$ is called 
a circle. 
For such $r,s$, 
we define 
$A_K [r,s)$ and $A_K (r,s]$ similarly, 
which we call semi-open annuli. 

\begin{defn}
A wide open rigid curve over $K$ is 
a one-dimensional smooth geometrically connected rigid space 
$W$ over $K$ which contains affinoid subdomains 
$X$ and $Y$ such that 
\begin{enumerate}
\item 
$W \setminus X$ is a disjoint union of finitely many open annuli, 
\item 
$X$ is relatively compact in $Y$ over $K$ (cf.\ \cite[9.6.2]{BGR}), 
\item 
$Y \cap V$ is a semi-open annulus for any connected component 
$V$ of $W \setminus X$. 
\end{enumerate}
We call $X$ an underlying affinoid of $W$. 
\end{defn}

\begin{thm}\cite[Theorem 2.18]{CMp^3}\label{cptif}
Let $W$ be a wide open rigid curve over $K$ 
with an underlying affinoid $X$. 
Then $W$ may be completed to 
a proper algebraic curve $C$ over $K$ 
by gluing closed disks to the connected components of 
$W \setminus X$. 
\end{thm}

Let $X$ be a rigid space over $K$. 
We write $A(X)$ for $\mathcal{O}_X (X)$. 
We put 
\begin{align*}
 |f|_{\mathrm{sup}} &= \sup_{x \in X(\mathbf{C} )} |f(x)| 
 \quad \textrm{for} \ f \in A(X), \\ 
 A^{\circ} (X) &=\Bigl\{ f \in A(X) \ \Bigm| \ 
 |f|_{\mathrm{sup}} \leq 1 \Bigr\}, \\
 A^{\circ \circ} (X) &=\Bigl\{ f \in A(X) \ \Bigm| \ 
 |f(x)| < 1 \textrm{ for all }x \in X(\mathbf{C}) \Bigr\}. 
\end{align*}
Then $A^{\circ} (X)$ is called the ring of 
bounded rigid analytic functions on $X$. 
We consider 
$A^{\circ} (X)$ as a linearly topologized ring 
with the ideal of definition 
$A^{\circ \circ} (X)$. 

Let $X$ be an affinoid rigid space over $K$. 
We write $||f||_X$ instead of 
$|f|_{\mathrm{sup}}$ for $f \in A(X)$. 
We put 
\[
 \overline{X} =\Spec \bigl( 
 A^{\circ} (X)/A^{\circ \circ} (X) \bigr), 
\]
which we call the reduction of $X$. 
The reduction $\overline{X}$ 
is of finite type over $k$ by \cite[6.3.4. Corollary 3]{BGR}. 
The canonical reduction map of $X$ 
is denoted by 
$\mathrm{Red}_X \colon X \to \overline{X}$. 

\begin{defn}
A basic wide open pair of rigid curves is a pair $(W,X)$ 
where $W$ is a wide open rigid curve over $K$ 
and $X$ is an affinoid subdomain of $W$ 
such that 
\begin{enumerate}
\item 
$\overline{X}$ is an irreducible curve 
with at worst ordinary double points as singularities, 
\item 
$||A(X)||_X =|K|$, 
\item 
each connected component of $W \setminus X$ is 
isomorphic to an annulus of the form 
$A_K (r,1)$ for $r \in |K^{\times} |$ satisfying $r<1$. 
\end{enumerate}
We say that $W$ is a basic wide open rigid curve, 
if $(W,X)$ is a basic wide open pair for some $X$. 
\end{defn}

If $(W,X)$ is a basic wide open pair of rigid curves, 
then $\overline{X}^{\mathrm{c}}$ denotes 
the compactification of $\overline{X}$ 
that is smooth at the cusps, 
where the cusps mean the points in 
$\overline{X}^{\mathrm{c}} \setminus \overline{X}$. 

\begin{defn}
Let $C$ be a wide open rigid curve or 
a proper smooth curve over $K$. 
A semi-stable covering of $C$ over $K$ is 
a finite set $\mathcal{S}$ 
of basic wide open pairs $(U,U^\mathrm{u})$ such that
\begin{enumerate}
\item 
$\mathcal{S}^{\mathrm{w}} =\{ U \mid (U,U^\mathrm{u}) \in \mathcal{S} \}$ 
is an admissible covering of $C$,
\item 
if $U_1,U_2 \in \mathcal{S}^{\mathrm{w}}$ and 
$U_1 \neq U_2$, then 
$U_1 \cap U_2$ is a disjoint union of connected components 
of $U_1 \setminus U_1 ^{\mathrm{u}}$, 
\item 
if $U_1$, $U_2$ and $U_3$ are three distinct elements of 
$\mathcal{S}^{\mathrm{w}}$, then 
$U_1 \cap U_2 \cap U_3 = \emptyset$. 
\end{enumerate}
\end{defn}

\begin{thm}\cite[Theorem 2.36]{CMp^3}\label{constructsst}
Let $C$ be a proper smooth curve over $K$. 
If $C$ has a semi-stable covering over $K$, 
then $C$ has an associated 
semi-stable model over $\mathcal{O}_K$ 
whose reduction has at least two irreducible components. 
\end{thm}

\begin{rem}
In \cite[Theorem 2.36]{CMp^3}, it is assumed that 
$K$ satisfies Hypothesis T, 
which means that $\mathbf{C}$ is isomorphic to the completion of 
an algebraic closure of a non-archimedean local field. 
The field $K$ in this paper satisfies Hypothesis T. 
\end{rem}

\section{Morphism between graphs}
In this paper, 
a graph means a finite directed graph 
such that 
each edge has two directions. 

Let $\Gamma$ be a graph. 
The set of the vertices of $\Gamma$ is 
denoted by $\mathcal{V}(\Gamma)$, 
and the set of the directed edges of $\Gamma$ is 
denoted by $\mathcal{E}(\Gamma)$. 
For $e \in \mathcal{E}(\Gamma)$, 
the source of $e$ and the target of $e$ 
are denoted by 
$s(e)$ and $t(e)$ respectively. 
For $e \in \mathcal{E}(\Gamma)$, 
let $\overline{e}$ be 
the directed edge obtained by reversing the direction 
of $e$. 

We take a prime number $\ell$ that is different from $p$. 
Let 
$V(\Gamma ,\mathbb{Q}_{\ell} )$ 
be the $\mathbb{Q}_{\ell}$-vector space 
generated by $\mathcal{V}(\Gamma)$, 
and let $E(\Gamma ,\mathbb{Q}_{\ell} )$ 
be the $\mathbb{Q}_{\ell}$-vector space 
generated by $\mathcal{E}(\Gamma)$ 
with relation $e =-\overline{e}$ for all 
$e \in \mathcal{E}(\Gamma)$. 
We consider two $\mathbb{Q}_{\ell}$-linear maps 
\begin{align*}
d& \colon E(\Gamma ,\mathbb{Q}_{\ell} )
 \longrightarrow V(\Gamma ,\mathbb{Q}_{\ell} );\ 
 e \mapsto t(e) -s(e), \\
\delta& \colon V(\Gamma ,\mathbb{Q}_{\ell} )
 \longrightarrow E(\Gamma ,\mathbb{Q}_{\ell} );\ 
 v \mapsto \sum_{e \in \mathcal{E}(\Gamma),\ t(e)=v} e. 
\end{align*}
We put 
$H_1 (\Gamma, \mathbb{Q}_{\ell}) =\Ker(d)$ 
and 
$H^1 (\Gamma, \mathbb{Q}_{\ell}) =\Coker(\delta)$. 
A cycle $R$ in $\Gamma$ can be considered 
as an element of $H_1 (\Gamma, \mathbb{Q}_{\ell})$. 
We consider a natural bilinear pairing
\[
 \langle\ ,\ \rangle \colon E(\Gamma ,\mathbb{Q}_{\ell} ) \times 
 E(\Gamma ,\mathbb{Q}_{\ell} ) \longrightarrow 
 \mathbb{Q}_{\ell} 
\]
determined by 
\begin{equation*}
 \langle e_1 ,e_2 \rangle=
 \begin{cases}
 1 & \textrm{if } e_1 =e_2 \\
 0 & \textrm{if } e_1 \neq e_2 ,\overline{e}_2 
 \end{cases}
\end{equation*}
for $e_1 ,e_2 \in \mathcal{E}(\Gamma)$. 
Then this pairing induces a bilinear perfect pairing 
\[
 H_1 (\Gamma, \mathbb{Q}_{\ell}) \times 
 H^1 (\Gamma, \mathbb{Q}_{\ell}) \longrightarrow 
 \mathbb{Q}_{\ell}. 
\]
Therefore we have a canonical isomorphism 
$H^1 (\Gamma, \mathbb{Q}_{\ell}) \cong 
 H_1 (\Gamma, \mathbb{Q}_{\ell})^*$. 

Let $\Gamma_1$ and $\Gamma_2$ be graphs. 

\begin{defn}
A finite flat morphism $\phi \colon \Gamma_1 \to \Gamma_2$ 
of degree $n$ 
consists of the following data: 
\begin{itemize}
\item 
Surjective maps 
$\phi_V \colon \mathcal{V}(\Gamma_1) \to \mathcal{V}(\Gamma_2)$ 
and 
$\phi_E \colon \mathcal{E}(\Gamma_1) \to \mathcal{E}(\Gamma_2)$ 
such that 
$\phi_V \circ s =s \circ \phi_E$, 
$\phi_E (\overline{e}) =\overline{\phi_E (e)}$ 
for $e \in \mathcal{E}(\Gamma_1)$, 
and the map 
$s \colon \phi_E ^{-1} (e') \to \phi_V ^{-1} \bigl( s(e') \bigr)$ is 
surjective for $e' \in \mathcal{E}(\Gamma_2 )$. 
\item
Positive integers $n_v$ and $n_e$ for 
all $v \in \mathcal{V}(\Gamma_1 )$ 
and $e \in \mathcal{E}(\Gamma_1 )$ such that 
$n_e =n_{\overline{e}}$ for $e \in \mathcal{E}(\Gamma_1 )$, 
\[
 \sum_{e \in \phi_E ^{-1} (e')} n_e =n \quad \ \textrm{and} \quad 
 \sum_{e \in \phi_E ^{-1} (e'),\, s(e)=v} n_e =n_v
\] 
for $e' \in \mathcal{E}(\Gamma_2 )$ and 
$v \in \phi_V ^{-1} \bigl( s(e') \bigr)$. 
\end{itemize}
\end{defn}

Let $\phi \colon \Gamma_1 \to \Gamma_2$ be 
a finite flat morphism of degree $n$. 
For a cycle $R=e_1 \cdots e_m$ in $\Gamma_1$, 
the cycle $\phi_E (e_1 )\cdots \phi_E (e_m )$ in $\Gamma_2$ 
is denoted by $\phi (R)$. 

\begin{prop}\label{pullcycle}
Let $R'$ be 
a cycle in $\Gamma_2$. 
Then there are cycles 
$R_1 ,\ldots , R_m$ in $\Gamma_1$ such that 
$\phi (R_i )=R'$ for $1 \leq i \leq m$ 
and 
\[
 \bigl| \{ 1 \leq i \leq m \mid e \in \mathcal{E}(R_i ) \} \bigr| 
 =n_e 
\]
for all $e \in \mathcal{E}(\Gamma_1 )$ 
satisfying $\phi_E (e) \in \mathcal{E}(R' )$. 
Furthermore, 
$\sum_{i=1} ^m R_i \in 
 H_1 (\Gamma_1, \mathbb{Q}_{\ell})$ does 
not depend on a choice of $R_1 ,\ldots , R_m$. 
\end{prop}
\begin{proof}
By replacing $\Gamma_1$ and $\Gamma_2$ by 
their subgraphs, 
we may assume $\Gamma_2 =R'$. 

We prove the first claim by induction on $n$. 
If $n=1$, the claim is trivial. 
We assume $n\geq 2$. 
By the surjectivity of 
$s \colon \phi_E ^{-1} (e') \to \phi_V ^{-1} \bigl( s(e') \bigr)$ 
for $e' \in \mathcal{E}(\Gamma_2 )$, 
we can easily find a cycle 
$R$ in $\Gamma_1$ such that 
$\phi (R)=R'$. 
Then we put 
\[
 n_v '=
 \begin{cases}
 n_v -1 & \textrm{if } v \in \mathcal{V}(R) \\ 
 n_v & \textrm{if } v \notin \mathcal{V}(R)
 \end{cases} 
 \quad \textrm{and} \quad 
 n_e ' =
 \begin{cases}
 n_e -1 & \textrm{if } e \in \mathcal{E}(R) \\ 
 n_e & \textrm{if } e \notin \mathcal{E}(R) 
 \end{cases} 
\]
for $v \in \mathcal{V}(\Gamma_1 )$ and 
$e \in \mathcal{E}(\Gamma_1 )$. 
We consider the subgraph $\Gamma_1 '$ 
of $\Gamma_1$ obtained from $\Gamma_1$ by 
removing $v \in \mathcal{V}(\Gamma_1 )$ and 
$e \in \mathcal{E}(\Gamma_1 )$ 
such that $n_v '=0$ and $n_e '=0$. 
We define a positive integer 
$\deg (R/R')$ by 
\[
 \phi (R)= \deg (R/R') R' \in H_1 (\Gamma_2 ,\mathbb{Q}_{\ell} ). 
\]
Then the restriction of $\phi$ gives 
a finite flat morphism
$\phi' \colon \Gamma_1 ' \to \Gamma_2$ of degree 
$n-\deg (R/R')$ and 
it suffices to show the claim for $\phi'$. 
This follows from the induction hypothesis. 

The last claim follows from 
the condition 
$\bigl| \{ 1 \leq i \leq m \mid e \in \mathcal{E}(R_i ) \} 
 \bigr| =n_e $ for all $e \in \mathcal{E}(\Gamma_1 )$. 
\end{proof}

We define 
$\mathbb{Q}_{\ell}$-linear maps 
$\phi_* \colon 
 H_1 (\Gamma_1, \mathbb{Q}_{\ell})
 \to H_1 (\Gamma_2, \mathbb{Q}_{\ell})$ 
and 
$\phi^* \colon 
 H_1 (\Gamma_2, \mathbb{Q}_{\ell})
 \to H_1 (\Gamma_1, \mathbb{Q}_{\ell})$ 
by 
$\phi_* (R) =\phi (R)$ for a cycle $R$ in $\Gamma_1$ and 
$\phi^* (R')=\sum_{i=1} ^m R_i$ 
for a cycle $R'$ in $\Gamma_2$, 
where 
$R_1 ,\ldots , R_m$ are 
as in Proposition \ref{pullcycle}. 

We define 
$\phi_* \colon H^1 (\Gamma_1, \mathbb{Q}_{\ell})
 \to H^1 (\Gamma_2, \mathbb{Q}_{\ell})$ 
and 
$\phi^* \colon 
 H^1 (\Gamma_2, \mathbb{Q}_{\ell})
 \to H^1 (\Gamma_1, \mathbb{Q}_{\ell})$ 
as the dual $\mathbb{Q}_{\ell}$-linear maps of 
$\phi^* \colon 
 H_1 (\Gamma_2, \mathbb{Q}_{\ell})
 \to H_1 (\Gamma_1, \mathbb{Q}_{\ell})$ 
and 
$\phi_* \colon H_1 (\Gamma_1, \mathbb{Q}_{\ell})
 \to H_1 (\Gamma_2, \mathbb{Q}_{\ell})$ 
respectively. 

\section{Formal model of rigid curve}
Let $W$ be a wide open rigid curve with 
a semi-stable covering 
$\mathcal{S} =\{ (U_i ,U_i ^\mathrm{u} ) 
 \mid i \in I \}$ over $K$. 
In this section, we construct a formal model of 
$W$ from the semi-stable covering $\mathcal{S}$. 
First, we recall some facts from 
\cite{CMp^3}. 

\begin{prop}\cite[Proposition 2.21]{CMp^3}\label{affi}
Let $C$ be a proper smooth curve over $K$. 
Let $L$ be a finite Galois extension of $K$, 
and let $T$ be a finite nonempty Galois stable 
subset of $C(L)$. 
Suppose that $\{ D_t \mid t \in T \}$ is a 
Galois stable collection of disjoint open disks 
over $L$ in $C$ such that 
$D_t \cap T=\{ t \}$ for all $t \in T$. 
Then $C \setminus (\bigcup_{t \in T} D_t )$ 
is a one-dimensional affinoid over $K$. 
\end{prop}

\begin{lem}\cite[Lemma 2.24]{CMp^3}\label{morph}
Let $f \colon X \to Y$ be a morphism between 
smooth one-dimensional affinoid rigid spaces over $K$. 
We assume that $\overline{X}$ is irreducible. 
\begin{enumerate}
\item[(i)] 
If 
$\overline{f} \colon \overline{X} \to \overline{Y}$ is 
a surjection, then $f$ is a surjection. 
\item[(ii)] 
If $\overline{f} (\overline{X}) \subset \overline{Y}$ 
is an open affine subset and 
$f \colon X(\mathbf{C} ) \to Y(\mathbf{C})$ is an injection, 
then $\overline{f}$ is an injection. 
\end{enumerate}
\end{lem}

\begin{lem}\cite[Proposition 2.8]{CMp^3}\label{sst}
Let $X$ be a smooth one-dimensional affinoid rigid curve 
such that $||A(X)||_X =|K|$. 
We assume that $P \in \overline{X}(\overline{k})$ and 
$\mathrm{Red}_X^{-1} (P)$ is isomorphic to an open annulus over $K$. 
Then $P$ is an ordinary double point of $\overline{X}$. 
\end{lem}

For $i \in I$, let $\{ V_j \mid j \in J_i \}$ be the 
set of the connected components of 
$U_i \setminus (U_i ^\mathrm{u} \cup \bigcup_{i' \in I, i' \neq i} U_{i'} )$. 
For different $i_1 ,i_2 \in I$, 
let $\{ V_j \mid j \in J_{i_1 ,i_2 } \}$ be the 
set of the connected components of 
$U_{i_1} \cap U_{i_2}$. 
For $i \in I$ and $j \in J_i$, 
let $V_j ^{\mathrm{s}}$ be the semi-open annulus 
obtained by adding a circle $C_j$ to 
$V_j$ from the opposite side to $U_i ^\mathrm{u}$. 
We define a rigid space $W'$ 
as $W \cup \bigcup_{i \in I, j \in J_i} V_j ^{\mathrm{s}}$. 
We put $\widetilde{I} =I \cup \bigcup_{i \in I} J_i$, 
and 
\[
 (U_i ,U_i ^\mathrm{u} )= 
\begin{cases}
 (U_i ,U_i ^\mathrm{u} ) & \textrm{if } i \in I \\
 (V_i ^{\mathrm{s}} ,C_i ) & \textrm{if } i' \in I \textrm{ and } i \in J_{i'} 
\end{cases}
\]
for $i \in \widetilde{I}$. 
We put 
\[
 J_{i_1 ,i_2} = 
\begin{cases}
 J_{i_1 ,i_2} & \textrm{if } i_1, i_2 \in I \\ 
 \{ i_1 \} & \textrm{if } i_2 \in I \textrm{ and } i_1 \in J_{i_2} \\ 
 \{ i_2 \} & \textrm{if } i_1 \in I \textrm{ and } i_2 \in J_{i_1} \\ 
 \emptyset & \textrm{otherwise } 
\end{cases}
\]
for different $i_1 ,i_2 \in \widetilde{I}$. 

\begin{prop}\label{compati} 
For $i \in \widetilde{I}$, 
let $S_i$ be a nonempty $G_k$-stable collection of 
points of the smooth locus of $\overline{U_i ^{\mathrm{u}}}$. 
We put 
$U(S_i )=U_i ^\mathrm{u} \setminus 
 \bigl( \bigcup_{s \in S_i} \mathrm{Red}_{U_i ^\mathrm{u}} ^{-1} (s) \bigr)$ 
for $i \in \widetilde{I}$
and 
$X_j (S_{i_1} ,S_{i_2} )= 
 U (S_{i_1} ) \cup U (S_{i_2} ) \cup V_j$ 
for different $i_1 ,i_2 \in \widetilde{I}$ and $j \in J_{i_1 ,i_2}$. 
Then $U (S_i )$ and 
$X_j (S_{i_1} ,S_{i_2} )$ are 
affinoid rigid spaces over $K$, and 
$\mathrm{Red}_{X_j (S_{i_1} ,S_{i_2} )}$ 
is compatible with 
$\mathrm{Red}_{U (S_{i_1} )}$ and 
$\mathrm{Red}_{U (S_{i_2} )}$ 
for $i \in \widetilde{I}$ and 
different $i_1 ,i_2 \in \widetilde{I}$ 
and $j \in J_{i_1 ,i_2}$. 
\end{prop}
\begin{proof}
We put 
$W' (\{ S_i \}_{i \in \widetilde{I}} ) =
 W' \setminus 
 \bigl( \bigcup_{s \in S_i ,i \in \widetilde{I}} 
 \mathrm{Red}_{U_i ^\mathrm{u}} ^{-1} (s) \bigr)$. 
Let $C$ be a proper smooth curve over $K$ obtained from 
$W$ as in Theorem \ref{cptif}. 
Then $W' (\{ S_i \}_{i \in \widetilde{I}} )$ 
is obtained from $C$ 
removing a $G_K$-stable union of disjoint open disks. 
Therefore 
$W' (\{ S_i \}_{i \in \widetilde{I}} )$ 
is an affinoid rigid space over $K$ by Proposition \ref{affi}. 

We take $i \in \widetilde{I}$. 
We know that 
$U (S_i )$ is 
an affinoid rigid space 
by \cite[Lemma 4.8.1.(a)]{FvdP}. 
The natural inclusion map 
$j_{U(S_i )} \colon U(S_i ) \to 
 W' (\{ S_i \}_{i \in \widetilde{I}} )$ 
induces a map on the reduction 
$\overline{j}_{U(S_i )} \colon \overline{U}(S_i ) \to 
 \overline{W'} (\{ S_i \}_{i \in \widetilde{I}} )$. 
Let $\mathrm{Im}(\overline{j}_{U(S_i )})$ be 
the image of $\overline{j}_{U(S_i )}$. 
Then $\mathrm{Im}(\overline{j}_{U(S_i )})$ is 
a point or an affine open subset of 
$\overline{W'} (\{ S_i \}_{i \in \widetilde{I}} )$. 

We assume that 
$\mathrm{Im}(\overline{j}_{U(S_i )})$ is a point $P$. 
Then 
$\mathrm{Red}_{W' (\{ S_i \}_{i \in \widetilde{I}} )}
 ^{-1} (P)$ 
is not connected to 
$\bigcup_{s \in S_i} \mathrm{Red}_{U_i ^\mathrm{u}} ^{-1} (s)$. 
Therefore 
$U(S_i ) \subset 
 \mathrm{Red}_{W' (\{ S_i \}_{i \in \widetilde{I}} )} 
 ^{-1} (P)$ 
is not connected to 
$\bigcup_{s \in S_i} \mathrm{Red}_{U_i ^\mathrm{u}} ^{-1} (s)$. 
This is a contradiction. 
Thus we have proved that 
$\mathrm{Im}(\overline{j}_{U(S_i )})$ is 
an affine open subset of 
$\overline{W'} (\{ S_i \}_{i \in \widetilde{I}} )$. 
Then the map 
$\overline{j}_{U(S_i )}$ is an injection by Lemma \ref{morph} (ii). 
Further, we have that 
$U(S_i )= 
 \mathrm{Red}_{W' (\{ S_i \}_{i \in \widetilde{I}} )} 
 ^{-1} \bigl( \mathrm{Im}(\overline{j}_{U(S_i )}) \bigr)$ 
by Lemma \ref{morph} (i). 
Let $Y(S_i )$ be the irreducible component of 
$\overline{W'} (\{ S_i \}_{i \in \widetilde{I}} )$ 
that contains 
$\mathrm{Im}(\overline{j}_{U(S_i )})$. 

We take $j \in J_{i_1 ,i_2}$ for different 
$i_1 ,i_2 \in \widetilde{I}$. 
Then 
$V_j = \mathrm{Red}_{W' (\{ S_i \}_{i \in \widetilde{I}} )}^{-1} ( y_j )$ 
for a point $y_j$ of 
$\overline{W'} (\{ S_i \}_{i \in \widetilde{I}} )$ 
by the connectedness of $V_j$. 
Further, we have that $y_j \in Y(S_{i_1} ) \cap Y(S_{i_2} )$, 
because $V_j$ is connected to 
$U(S_{i_1} )$ and $U(S_{i_2} )$. 

Therefore we have 
\[ 
 Y(S_i ) = \mathrm{Im}(\overline{j}_{U(S_i )}) 
 \cup \bigcup_{j \in J_{i,i'},\ i' \in \widetilde{I}} \{ y_j \}
\] 
for all $i \in \widetilde{I}$. 
Then 
$X_j (S_{i_1} ,S_{i_2} ) = 
 \mathrm{Red}_{W' (\{ S_i \}_{i \in \widetilde{I}} )} 
 ^{-1} \bigl( \mathrm{Im}(\overline{j}_{U(S_{i_1} )}) \cup 
 \mathrm{Im}(\overline{j}_{U(S_{i_2} )}) \cup \{ y_j \} \bigr)$ 
is an affinoid rigid space by 
\cite[Lemma 4.8.1.(a)]{FvdP}. 
The compatibility also follows from 
\cite[Lemma 4.8.1.(c)]{FvdP}. 
\end{proof}

For a formal scheme $\mathscr{X}$ over 
$\Spf \mathcal{O}_K$, 
the closed fiber $\mathscr{X}_k$ of $\mathscr{X}$ 
means the underlying reduced scheme of 
the ringed space 
$(\mathscr{X} ,\mathcal{O}_{\mathscr{X}} /\mathcal{I})$, 
where $\mathcal{I}$ is an ideal of definition of 
$\mathscr{X}$. 

\begin{thm}\label{const} 
Let $W$ be a wide open rigid curve with 
a semi-stable covering 
$\mathcal{S} =\{ (U_i ,U_i ^\mathrm{u} ) 
 \mid i \in I \}$ over $K$. 
Then $W$ has an associated 
semi-stable formal model over $\mathcal{O}_K$. 
\end{thm}
\begin{proof}
We take $S_i$ for $i \in \widetilde{I}$ 
as in Proposition \ref{compati}. 
Then 
$\{ U_i ^{\mathrm{u}} \}_{i \in \widetilde{I}} \cup 
 \{ X_j (S_{i_1} ,S_{i_2} )
 \}_{i_1 ,i_2 \in \widetilde{I} ,j \in J_{i_1 ,i_2} }$ 
is a pure affinoid covering of $W'$ in the sense of 
\cite[Definition 4.8.3]{FvdP} by 
Lemma \ref{compati}. 
This covering gives a formal model 
$\mathscr{W}'$ of $W'$. 
The closed fiber $\mathscr{W}' _k$ is 
a semi-stable curve by Lemma \ref{sst}. 

Let $Y$ be the reduced closed subscheme of 
$\mathscr{W}' _k$ determined by 
$\bigcup_{i \in I} \overline{U_i ^{\mathrm{u}}}^{\mathrm{c}}$. 
We define $\mathscr{W}$ as 
the formal completion of $\mathscr{W}'$ 
along $Y$. 
Then $\mathscr{W}$ is a formal model 
of $W$. 
We can check that $\mathscr{W}$ is 
independent of a choice of 
$S_i$ for $i \in \widetilde{I}$. 
\end{proof}

We use the notation in the proof of 
Theorem \ref{const}. 
The closed fiber $\mathscr{W}_k$ 
of $\mathscr{W}$ is $Y$. 
The irreducible components of the 
geometric closed fiber 
$\mathscr{W'}_{\overline{k}}$ of $\mathscr{W'}$ 
consists of proper curves that are 
also irreducible components of 
$Y_{\overline{k}}$ 
and affine lines over 
$\overline{k}$. 

\begin{prop}\label{patch}
For $i \in I$ and $j \in J_i$, 
we put 
$X_j ' (S_i )=U(S_i ) \cup V_j$. 
Then $\mathscr{W}$ is obtained also by patching 
$\Spf A^{\circ} (U_i ^{\mathrm{u}} )$ 
for $i \in I$, 
$\Spf A^{\circ} \bigl( X_j (S_{i_1} ,S_{i_2} ) \bigr)$ 
for 
$i_1 ,i_2 \in I$ and $j \in J_{i_1 ,i_2}$, 
and $\Spf A^{\circ} \bigl( X_j ' (S_i ) \bigr)$ 
for $i \in I$ and $j \in J_i$. 
\end{prop}
\begin{proof}
We need to show that 
the ring of an affine open formal subscheme of 
$\mathscr{W}$ can be recovered as 
the ring of bounded rigid analytic functions on its rigid generic fiber. 
This follows from 
\cite[Theorem 7.4.1]{deJcryrigid}. 
\end{proof}

Let $C$ be the proper smooth curve over $K$ 
obtained from $W$ as in Theorem \ref{cptif}. 
Let $\{ D_{i'} \}_{i' \in I'}$ be the set of 
the connected components of $D=C \setminus W$. 
For $i' \in I'$, there uniquely exist 
$i \in I$ and $j \in J_i$ such that 
the union of $D_{i'}$ and $V_j$ 
defines an open disk in $C$, which is denoted by 
$U_{i'}$. 
Then 
$\mathcal{S}' = \mathcal{S} \cup 
 \{ (U_{i'} ,D_{i'} ) \mid i' \in I' \}$ is 
a semi-stable covering of $C$. 
Let $\mathcal{C}$ be the semi-stable model of $C$ 
associated to the semi-stable covering $\mathcal{S}'$ 
by Theorem \ref{constructsst}. 
Then $Y$ 
is naturally considered as 
a closed subscheme of the special fiber 
$\mathcal{C}_k$ of $\mathcal{C}$. 

\begin{prop}\label{sstformal}
The formal completion of $\mathcal{C}$ 
along $Y$ is naturally 
isomorphic to $\mathscr{W}$. 
\end{prop}
\begin{proof}
This follows from the construction of 
$\mathscr{W}$ in the proof of Theorem \ref{const}. 
\end{proof}

\begin{rem}\label{inducerem}
Proposition \ref{patch} 
and Proposition \ref{sstformal} 
give two different descriptions 
of the same formal model $\mathscr{W}$, 
and the both descriptions are important. 
The construction in Proposition \ref{patch} 
implies that 
a finite flat morphism between 
wide open rigid curves 
which is compatible with their semi-stable 
coverings naturally extends to 
a morphism between their semi-stable formal models. 
This fact is very non-trivial from the construction 
in Proposition \ref{sstformal}, 
because such a morphism does not extend to a morphism 
between their compactifications in general. 
On the other hand, 
by Proposition \ref{sstformal}, 
we see that $W$ satisfies the condition of 
\cite[Proposition 5.9.4]{FaCohLLC}. 
\end{rem}

\section{Cohomology of rigid curve}
We put 
$\mathscr{W}_{\mathcal{O}_{\mathbf{C}}} =
 \mathscr{W} \widehat{\otimes}_{\mathcal{O}_K} 
 \mathcal{O}_{\mathbf{C}}$. 
The formal nearby cycle functor 
$R\varPsi_{\mathscr{W}_{\mathcal{O}_{\mathbf{C}}}}$ of 
$\mathscr{W}_{\mathcal{O}_{\mathbf{C}}}$ is defined 
in \cite[section 2]{BeVanII}. 

Let $\Gamma_{\mathscr{W}}$ 
and $\Gamma_{\mathscr{W}'}$ be the 
dual graphs of 
$\mathscr{W}_{\overline{k}}$ and 
$\mathscr{W}'_{\overline{k}}$ respectively 
(cf. \cite[Definition 10.3.17]{LiuAlg} ). 
Then $\Gamma_{\mathscr{W}}$ 
is a subgraph of 
$\Gamma_{\mathscr{W}'}$. 
For $v \in \mathcal{V}(\Gamma_{\mathscr{W}'} )$, 
let $Y_v$ be the irreducible component of 
$\mathscr{W}'_{\overline{k}}$ 
corresponding to $v$, 
let $\widetilde{Y}_v$ be the normalization of 
$Y_v$, and let 
$\pi_v \colon \widetilde{Y}_v \to \mathscr{W}'_{\overline{k}}$ 
be a natural morphism. 

\begin{prop}\label{nearby} 
Let $\Lambda$ be a torsion local finite 
$\mathbb{Z}_{\ell}$-algebra. 
Then there are canonical isomorphisms 
\begin{align}
 R^0 \varPsi_{\mathscr{W}_{\mathcal{O}_{\mathbf{C}}}} 
 \Lambda 
 &\cong \Lambda, \label{R^0} \\ 
 R^1 \varPsi_{\mathscr{W}_{\mathcal{O}_{\mathbf{C}}}} 
 \Lambda 
 &\cong \biggl( \Bigl( 
 \bigoplus_{v \in \mathcal{V}(\Gamma_{\mathscr{W}'} )} 
 {\pi_v }_* \Lambda 
 \Bigr) \Big/ \Lambda \biggr)^* (-1) \label{R^1}, 
\end{align}
where we consider the right hand side of \eqref{R^1} 
as a sheaf on $\mathscr{W}_{\overline{k}}$. 
\end{prop}
\begin{proof}
The isomorphism 
\eqref{R^0} follows from the definition. 
By \cite[Theorem 3.1]{BeVanII}, 
we have an isomorphism 
\[
 R^1 \varPsi_{\mathscr{W}_{\mathcal{O}_{\mathbf{C}}}} 
 \Lambda \cong 
 (R^1 \varPsi_{\mathcal{C}_{\mathcal{O}_{\overline{K}}}} 
 \Lambda)|_{Y_{\overline{k}}}. 
\]
Hence it suffices to prove 
the isomorphism \eqref{R^1} 
locally at singular points of $\mathscr{W}'_{\overline{k}}$. 

Let $x_e$ be the singular point of 
$\mathscr{W}'_{\overline{k}}$ corresponding to 
$e \in \mathcal{E}(\Gamma_{\mathscr{W}'} )$. 
Then the formal completion $\mathscr{W}_e$ of 
$\mathcal{C}_{\mathcal{O}_{\overline{K}}}$ at $x_e$ 
is isomorphic to 
$\Spf \mathcal{O}_{\mathbf{C}} [[S,T]] /(ST-c)$ 
for some $c \neq 0 \in \mathfrak{m}_K$. 
Further, $\Spf \mathcal{O}_{\mathbf{C}} [[S,T]] /(ST-c)$ 
is isomorphic to the formal completion of 
$\mathcal{X} =\Spec \mathcal{O}_{\overline{K}} [S,T] /(ST-c)$ 
at the point $x_0$ of the special fiber 
defined by $S=T=0$. 
Note that 
$\mathcal{X}_{\overline{K}} \cong \mathbb{G}_{m,\overline{K}}$. 
Then we have isomorphisms 
\[
 (R^1 \varPsi_{\mathcal{C}_{\mathcal{O}_{\overline{K}}}} 
 \Lambda)|_{x_e} \cong 
 R^1 \varPsi_{\mathscr{W}_e} 
 \Lambda \cong 
 (R^1 \varPsi_{\mathcal{X}} 
 \Lambda)|_{x_0} \cong 
 H^1 (\mathbb{G}_{m,\overline{K}} ,\Lambda)
\]
by \cite[Theorem 3.1]{BeVanII} and 
\cite[Expos\'{e} XV Proposition 2.2.3]{SGA7II}. 
Let $i_0$ and $i_{\infty}$ be the 
closed immersions of the zero point and 
the infinity point 
into $\mathbb{P}^1$ respectively. 
Then we have 
\begin{align*}
 H^1 (\mathbb{G}_{m,\overline{K}} ,\Lambda) 
 &\cong 
 H^1 _{\mathrm{c}} (\mathbb{G}_{m,\overline{K}} ,\Lambda)^* (-1), \\ 
 H^1 _{\mathrm{c}} (\mathbb{G}_{m,\overline{K}} ,\Lambda) 
 &\cong 
 H^0 \bigl( \mathbb{P}^1_{\overline{K}} , 
 ({i_{0}}_* \Lambda \oplus {i_{\infty}}_* \Lambda )/\Lambda \bigr). 
\end{align*} 
Hence the claim follows, 
because the zero point and the infinity point 
correspond to the irreducible components passing 
$x_e$. 
\end{proof}

We put 
\[
 H^i (\mathscr{W}_{\overline{k}} , 
 R^j \varPsi_{\mathscr{W}_{\mathcal{O}_{\mathbf{C}}}}
 \mathbb{Q}_{\ell} )= 
 \biggl( \varprojlim_{N \in \mathbb{N} } 
 H^i \bigl( \mathscr{W}_{\overline{k}} , 
 R^j \varPsi_{\mathscr{W}_{\mathcal{O}_{\mathbf{C}}}}
 (\mathbb{Z} / \ell^N \mathbb{Z} ) \bigr) \biggr) 
 \otimes_{\mathbb{Z}_{\ell}} \mathbb{Q}_{\ell}. 
\] 
We consider $W_{\mathbf{C}}$ as 
a Berkovich space. 
Then we have a spectral sequence 
\[
 E_2^{i,j} =H^i (\mathscr{W}_{\overline{k}} , 
 R^j \varPsi_{\mathscr{W}_{\mathcal{O}_{\mathbf{C}}}}
 \mathbb{Q}_{\ell} )
 \Rightarrow 
 H^{i+j} (W_{\mathbf{C}}, \mathbb{Q}_{\ell} )
\]
by \cite[Proposition 5.9.4]{FaCohLLC}, 
because $W$ satisfies the condition in 
\cite[Proposition 5.9.4]{FaCohLLC} by 
Proposition \ref{sstformal}. 
This spectral sequence 
gives an exact sequence 
\begin{equation}\label{longN}
 0 \longrightarrow 
 H^1 (\mathscr{W}_{\overline{k}}, 
 \mathbb{Q}_{\ell} ) 
 \longrightarrow 
 H^1 (W_{\mathbf{C}} , \mathbb{Q}_{\ell}  ) 
 \longrightarrow 
 H^0 ( \mathscr{W}_{\overline{k}}, 
 R^1 \varPsi_{\mathscr{W}_{\mathcal{O}_{\mathbf{C}}}} 
 \mathbb{Q}_{\ell} ) \longrightarrow 
 H^2 ( \mathscr{W}_{\overline{k}}, 
 \mathbb{Q}_{\ell} ) 
\end{equation} 
by \eqref{R^0}. 

Let $\widetilde{\Gamma}_{\mathscr{W}}$ be the graph 
obtained by 
adding one new vertex to $\Gamma_{\mathscr{W}'}$ and 
joining the new vertex with all vertices in 
$\mathcal{V}(\Gamma_{\mathscr{W}'}) \setminus 
 \mathcal{V}(\Gamma_{\mathscr{W}})$. 

\begin{lem}\label{HandN}
There is a canonical isomorphism 
\[ 
 h \colon H_1 (\widetilde{\Gamma}_{\mathscr{W}} , 
 \mathbb{Q}_{\ell}) (-1) 
 \stackrel{\sim}{\longrightarrow} 
 \Ker \Bigl(  H^0 ( \mathscr{W}_{\overline{k}}, 
 R^1 \varPsi_{\mathscr{W}_{\mathcal{O}_{\mathbf{C}}}} 
 \mathbb{Q}_{\ell} ) \longrightarrow 
 H^2 (\mathscr{W}_{\overline{k}}, 
 \mathbb{Q}_{\ell} ) \Bigr) . 
\]
\end{lem}
\begin{proof}
By the isomorphism \eqref{R^1}, 
we have a canonical isomorphism 
\begin{equation}\label{cohR^1}
 H^0 ( \mathscr{W}_{\overline{k}}, 
 R^1 \varPsi_{\mathscr{W}_{\mathcal{O}_{\mathbf{C}}}} 
 \mathbb{Q}_{\ell} ) \cong 
 H^0 \biggl( \mathscr{W}_{\overline{k}}, 
 \Bigl( \bigoplus_{v \in \mathcal{V}(\Gamma_{\mathscr{W}'} )} 
 {\pi_v }_* \mathbb{Q}_{\ell} 
 \Bigr) \Big/ \mathbb{Q}_{\ell} \biggr)^* (-1). 
\end{equation}
Under the identification by \eqref{cohR^1}, 
we explain the definition of $h$. 
For $e \in \mathcal{E}(\Gamma_{\mathscr{W}'})$, 
let $P_e$ be the point of 
$Y_{s(e)} \cap Y_{t(e)}$ that 
corresponds to $e$, 
and define $c_e$ as an element
\[
 (0,1) \in 
 \bigl( ( {\pi_{s(e)}}_* \mathbb{Q}_{\ell} 
 \oplus 
 {\pi_{t(e)}}_* \mathbb{Q}_{\ell} 
 ) / \mathbb{Q}_{\ell} \bigr)_{P_e} 
 \subset 
 H^0 \biggl( \mathscr{W}_{\overline{k}}, 
 \Bigl( \bigoplus_{v \in \mathcal{V}(\Gamma_{\mathscr{W}'} )} 
 {\pi_v }_* \mathbb{Q}_{\ell} 
 \Bigr) \Big/ \mathbb{Q}_{\ell} \biggr). 
\]
Then we define 
a $\mathbb{Q}_{\ell}$-linear map 
$h$ by 
\[
 h(r)(c_e )=\langle r, e\rangle
\]
for $r \in H_1 (\widetilde{\Gamma}_{\mathscr{W}} , 
 \mathbb{Q}_{\ell}) 
 \subset E (\widetilde{\Gamma}_{\mathscr{W}} , 
 \mathbb{Q}_{\ell})$ and 
$e \in \mathcal{E}(\Gamma_{\mathscr{W}'})$, 
where $e$ is considered as an element of 
$E (\Gamma_{\mathscr{W}'} , 
 \mathbb{Q}_{\ell}) 
 \subset E (\widetilde{\Gamma}_{\mathscr{W}} , 
 \mathbb{Q}_{\ell})$. 
Under the identification by \eqref{cohR^1}, 
the last map in \eqref{longN} is $(-1)$-twist of 
the dual of the natural map 
\[
 \bigoplus_{v \in \mathcal{V}(\Gamma_{\mathscr{W}} )} 
 \mathbb{Q}_{\ell} \longrightarrow 
 H^0 \biggl( \mathscr{W}_{\overline{k}}, 
 \Bigl( \bigoplus_{v \in \mathcal{V}(\Gamma_{\mathscr{W}'} )} 
 {\pi_v }_* \mathbb{Q}_{\ell} 
 \Bigr) \Big/ \mathbb{Q}_{\ell} \biggr). 
\]
Using this description, 
we can easily check that $h$ gives an isomorphism. 
\end{proof}

\begin{prop}\label{cohomology}
We have two exact sequences 
\begin{align*}
 &0 \longrightarrow 
 H^1 (\mathscr{W}_{\overline{k}} , 
 \mathbb{Q}_{\ell} ) \longrightarrow 
 H^1 (W_{\mathbf{C}}, \mathbb{Q}_{\ell} ) \longrightarrow 
 H_1 (\widetilde{\Gamma}_{\mathscr{W}} , 
 \mathbb{Q}_{\ell}) (-1) \longrightarrow 0, \\
 &0 \longrightarrow 
 H^1 (\Gamma_{\mathscr{W}} , 
 \mathbb{Q}_{\ell}) \longrightarrow 
 H^1 (\mathscr{W}_{\overline{k}} , 
 \mathbb{Q}_{\ell} ) \longrightarrow 
 \bigoplus_{v \in 
 \mathcal{V}(\Gamma_{\mathscr{W}} )} 
 H^1 (\widetilde{Y}_v , 
 \mathbb{Q}_{\ell})
 \longrightarrow 0. 
\end{align*}
\end{prop}
\begin{proof}
The first exact sequence follows from 
\eqref{longN} and Lemma \ref{HandN}. 

By a short exact sequence 
\[
 0 \longrightarrow \mathbb{Q}_{\ell} \longrightarrow 
 \bigoplus_{v \in \mathcal{V}(\Gamma_{\mathscr{W}} )} 
 {\pi_v }_* \mathbb{Q}_{\ell}
 \longrightarrow 
 \biggl( \bigoplus_{v \in \mathcal{V}(\Gamma_{\mathscr{W}} )} 
 {\pi_v }_* \mathbb{Q}_{\ell} 
 \biggr) \bigg/ \mathbb{Q}_{\ell} 
 \longrightarrow 0 
\]
on $\mathscr{W}_{\overline{k}}$, 
we have an exact sequence 
\begin{align*}
 \bigoplus_{v \in \mathcal{V}(\Gamma_{\mathscr{W}} )} 
 H^0 (\widetilde{Y}_v , \mathbb{Q}_{\ell}) 
 &\longrightarrow H^0 \biggl( \mathscr{W}_{\overline{k}}, 
 \Bigl( \bigoplus_{v \in \mathcal{V}(\Gamma_{\mathscr{W}} )} 
 {\pi_v }_* \mathbb{Q}_{\ell} 
 \Bigr) \Big/ \mathbb{Q}_{\ell} \biggr) \\ 
 &\longrightarrow H^1 ( \mathscr{W}_{\overline{k}}, 
 \mathbb{Q}_{\ell}) \longrightarrow 
 \bigoplus_{v \in \mathcal{V}(\Gamma_{\mathscr{W}} )} 
 H^1 (\widetilde{Y}_v , \mathbb{Q}_{\ell}) 
 \longrightarrow 0. 
\end{align*}
This exact sequence and a canonical isomorphism 
\[
 \Coker \Biggl( 
 \bigoplus_{v \in \mathcal{V}(\Gamma_{\mathscr{W}} )} 
 H^0 (\widetilde{Y}_v , \mathbb{Q}_{\ell}) 
 \longrightarrow H^0 \biggl( \mathscr{W}_{\overline{k}}, 
 \Bigl( \bigoplus_{v \in \mathcal{V}(\Gamma_{\mathscr{W}} )} 
 {\pi_v }_* \mathbb{Q}_{\ell} 
 \Bigr) \Big/ \mathbb{Q}_{\ell} \biggr) \Biggr) 
 \cong  H^1 (\Gamma_{\mathscr{W}} , 
 \mathbb{Q}_{\ell}) 
\]
gives the second exact sequence. 
\end{proof}

\begin{rem}
Proposition \ref{cohomology} can be considered as 
an explicit description of a part of 
the weight spectral sequence for 
$W_{\mathbf{C}}$ (cf.\ \cite[3.8]{Illmono}). 
\end{rem}

\section{Pushforward and pullback}
Let $W_1$ and $W_2$ be wide open rigid curves with 
semi-stable coverings 
$\mathcal{S}_1 =\{ (U_{1,i} ,U_{1,i} ^\mathrm{u}) 
 \mid i \in I_1 \}$
and 
$\mathcal{S}_2 =\{ (U_{2,i} ,U_{2,i} ^\mathrm{u}) 
 \mid i \in I_2 \}$ respectively. 
All the construction in the section $3$ 
applies to $W_1$ and $W_2$, 
and the subscripts $1$ and $2$ mean 
that it is constructed from $W_1$ and $W_2$ respectively. 

\begin{defn}
We say that 
a finite flat morphism 
$f \colon W_1 \to W_2$ is compatible with 
semi-stable coverings if, 
for any $i_1 \in I_1$, 
there is $i_2 \in I_2$ such that 
$f(U_{1,i_1} ^\mathrm{u})=U_{2,i_2} ^\mathrm{u}$ 
and 
$f(U_{1,i_1} \setminus U_{1,i_1} ^\mathrm{u} )
 =U_{2,i_2} \setminus U_{2,i_2} ^\mathrm{u}$. 
\end{defn}

Let $f \colon W_1 \to W_2$ be 
a finite flat morphism of degree $n$ 
that is compatible with 
semi-stable coverings. 
The morphism $f$ induces a finite morphism 
$\hat{f} \colon \mathscr{W}_1 \to \mathscr{W}_2$ 
by Proposition \ref{patch}. 
Further, 
$\hat{f}$ induces 
$\hat{f}_{\overline{k}} \colon \mathscr{W}_{1,\overline{k}} 
 \to \mathscr{W}_{2,\overline{k}}$ and 
a finite flat morphism 
$\phi_f \colon \Gamma_{\mathscr{W}'_1} 
 \to \Gamma_{\mathscr{W}'_2}$ 
of degree $n$. 
This induces 
a finite flat morphism
$\phi_f \colon \Gamma_{\mathscr{W}_1} 
 \to \Gamma_{\mathscr{W}_2}$ 
of degree $n$. 
The morphism $\phi_f$ naturally extends to a 
finite flat morphism 
$\widetilde{\phi}_f \colon \widetilde{\Gamma}_{\mathscr{W}_1} 
 \to \widetilde{\Gamma}_{\mathscr{W}_2}$ 
of degree $n$. 

In the remaining of this section, let $j=1,2$. 
We put 
$\mathscr{V}_j =\Spf \mathcal{O}_K [[ S_j ,T_j ]]/(S_j T_j -c_j )$ 
for some $c_j \neq 0 \in \mathfrak{m}_K$. 
Let $Y_j$ and $Y_j '$ be the closed subschemes of 
the geometric closed fiber $\mathscr{V}_{j,\overline{k}}$ 
defined by $T_j =0$ and $S_j =0$ respectively. 
We note that $Y_j = Y_j ' = \mathscr{V}_{j,\overline{k}}$. 
We put 
$\mathscr{V}_{j,\mathcal{O}_{\mathbf{C}}} =
 \mathscr{V}_j \widehat{\otimes}_{\mathcal{O}_K} 
 \mathcal{O}_{\mathbf{C}}$ and 
\[
 H^0 (\mathscr{V}_{j,\overline{k}} , 
 R^1 \varPsi_{\mathscr{V}_{j,\mathcal{O}_{\mathbf{C}}}} 
 \mathbb{Q}_{\ell} )= 
 \biggl( \varprojlim_{N \in \mathbb{N} } 
 H^0 \bigl( \mathscr{V}_{j,\overline{k}} , 
 R^1 \varPsi_{\mathscr{V}_{j,\mathcal{O}_{\mathbf{C}}}} 
 (\mathbb{Z} / \ell^N \mathbb{Z} ) \bigr) \biggr) 
 \otimes_{\mathbb{Z}_{\ell}} \mathbb{Q}_{\ell}. 
\] 
By \eqref{R^1} for $\mathscr{V}_j$, 
we have a canonical isomorphism 
\begin{equation}\label{discrN}
 H^0 (\mathscr{V}_{j,\overline{k}} , 
 R^1 \varPsi_{\mathscr{V}_{j,\mathcal{O}_{\mathbf{C}}}} 
 \mathbb{Q}_{\ell} )
 \cong 
 H^0 \bigl( \mathscr{V}_{j,\overline{k}} , 
 ({i_{Y_j}}_* \mathbb{Q}_{\ell} \oplus 
 {i_{Y_j '}}_* \mathbb{Q}_{\ell} ) / 
 \mathbb{Q}_{\ell} \bigr)^* (-1) , 
\end{equation}
where $i_{Y_j} \colon Y_j \to \mathscr{V}_{j,\overline{k}}$ 
and $i_{Y_j '} \colon Y_j ' \to \mathscr{V}_{j,\overline{k}}$ 
are identity morphisms. 
We fix an identification 
$\mathbb{Q}_{\ell} \cong \mathbb{Q}_{\ell} (1)$. 
Under the identifications \eqref{discrN} and 
$\mathbb{Q}_{\ell} \cong \mathbb{Q}_{\ell} (1)$, 
we define 
$\gamma_j \in H^0 (\mathscr{V}_{j,\overline{k}} , 
 R^1\varPsi_{\mathscr{V}_{j,\mathcal{O}_{\mathbf{C}}}} 
 \mathbb{Q}_{\ell} )$ 
by 
\[ 
 \gamma_j \bigl( (a,a' ) \bigr) =a-a' \quad 
 \textrm{ for } \ 
 (a,a' ) \in H^0 \bigl( \mathscr{V}_{j,\overline{k}}, 
 ({i_{Y_j}}_* \mathbb{Q}_{\ell} 
 \oplus 
 {i_{Y_j '}}_* \mathbb{Q}_{\ell} ) / 
 \mathbb{Q}_{\ell} \bigr). 
\]

\begin{lem}\label{loc} 
Let $V_j$ be the 
open annulus associated to $\mathscr{V}_j$. 
Let $g \colon \mathscr{V}_1 \to \mathscr{V}_2$ 
be a finite morphism such that 
the induced morphism $g \colon V_1 \to V_2$ 
is a finite flat morphism of degree $m$. 
We assume that 
$g \colon \mathscr{V}_1 \to \mathscr{V}_2$ 
induces a finite morphism 
$\Spf k[[S_1 ]] \to \Spf k[[S_2 ]]$. 
Let 
\[
 g_* \colon H^0 (\mathscr{V}_{1,\overline{k}} , 
 R^1\varPsi_{\mathscr{V}_{1,\mathcal{O}_{\mathbf{C}}}} 
 \mathbb{Q}_{\ell} ) \longrightarrow 
 H^0 (\mathscr{V}_{2,\overline{k}} , 
 R^1\varPsi_{\mathscr{V}_{2,\mathcal{O}_{\mathbf{C}}}} 
 \mathbb{Q}_{\ell} )
\] 
be the pushforward by $g$, and let 
\[
 g^* \colon H^0 (\mathscr{V}_{2,\overline{k}} , 
 R^1\varPsi_{\mathscr{V}_{2,\mathcal{O}_{\mathbf{C}}}} 
 \mathbb{Q}_{\ell} ) \longrightarrow 
 H^0 (\mathscr{V}_{1,\overline{k}} , 
 R^1\varPsi_{\mathscr{V}_{1,\mathcal{O}_{\mathbf{C}}}} 
 \mathbb{Q}_{\ell} ) 
\] 
be the pullback by $g$. 
Then we have $g_* (\gamma_1 )=\gamma_2$ and $g^* (\gamma_2 )=m\gamma_1$. 
\end{lem}
\begin{proof}
We have a canonical isomorphism 
\begin{equation}\label{cohnear}
 H^1 (V_{j,\mathbf{C}} ,\mathbb{Q}_{\ell} )
 \cong 
 H^0 (\mathscr{V}_{j,\overline{k}} , 
 R^1\varPsi_{\mathscr{V}_{j,\mathcal{O}_{\mathbf{C}}}} 
 \mathbb{Q}_{\ell} ) 
\end{equation}
by the similar exact sequence as \eqref{longN} for 
$\mathscr{V}_j$. 
Hence, we may argue on the left hand side of 
\eqref{cohnear}. 
Let $N$ be a positive integer. 
Using the long exact sequence obtained from 
\[
 0 \longrightarrow (\mathbb{Z}/\ell ^N \mathbb{Z} )(1) \longrightarrow 
 \mathcal{O}^{\times} _{V_{j,\mathbf{C}}} 
 \stackrel{\ell ^N}{\longrightarrow} 
 \mathcal{O}^{\times} _{V_{j,\mathbf{C}}} \longrightarrow 0, 
\]
we have an injection 
$\mathcal{O}^{\times} _{V_{j,\mathbf{C}}} /
 (\mathcal{O}^{\times} _{V_{j,\mathbf{C}}})^{\ell ^N} 
 \longrightarrow 
 H^1 \bigl( V_{j,\mathbf{C}} ,
 (\mathbb{Z}/\ell ^N \mathbb{Z} )(1) \bigr)$. 
We consider a scheme 
$\mathcal{X}_j = \Spec \mathcal{O}_K [ S_j ,T_j ]/(S_j T_j -c_j )$. 
We consider a commutative diagram 
\begin{align*}
 &\xymatrix{
 \mathcal{O}^{\times} _{\mathcal{X}_{j,\overline{K}}} /
 (\mathcal{O}^{\times} _{\mathcal{X}_{j,\overline{K}}})^{\ell ^N} (-1) 
 \ar@{->}[r]^{\sim} \ar@{->}[d] & 
 H^1 \bigl( \mathcal{X}_{j,\overline{K}} ,
 \mathbb{Z}/\ell ^N \mathbb{Z}  \bigr) 
 \ar@{->}[r]^{\sim \hspace*{2em}} \ar@{->}[d] & 
 H^0 \bigl( \mathcal{X}_{j,\overline{k}} , 
 R^1\varPsi_{\mathcal{X}_{j,\mathcal{O}_{\overline{K}}}} 
 (\mathbb{Z}/\ell ^N \mathbb{Z} ) \bigr) 
 \ar@{->}[d]^{\! \rotatebox{90}{$\sim$}} 
 \\
 \mathcal{O}^{\times} _{V_{j,\mathbf{C}}} /
 (\mathcal{O}^{\times} _{V_{j,\mathbf{C}}})^{\ell ^N} (-1) 
 \ar@{^{(}->}[r]  & 
 H^1 \bigl( V_{j,\mathbf{C}} , 
 \mathbb{Z}/\ell ^N \mathbb{Z} \bigr) 
 \ar@{->}[r]^{\sim \hspace*{2em}} & 
 H^0 \bigl( \mathscr{V}_{j,\overline{k}} , 
 R^1\varPsi_{\mathscr{V}_{j,\mathcal{O}_{\mathbf{C}}}} 
 (\mathbb{Z}/\ell ^N \mathbb{Z} ) \bigr), 
 }
\end{align*}
where the right vertical arrow is an isomorphism 
by \cite[Theorem 3.1]{BeVanII}. 
From this commutative diagram, 
we obtain an isomorphism 
\begin{equation}\label{Ocoh}
 \mathcal{O}^{\times} _{\mathcal{X}_{j,\overline{K}}} /
 (\mathcal{O}^{\times} _{\mathcal{X}_{j,\overline{K}}})^{\ell ^N} (-1) 
 \xrightarrow{\sim} 
 H^1 \bigl( V_{j,\mathbf{C}} ,
 \mathbb{Z}/\ell ^N \mathbb{Z} \bigr). 
\end{equation}

We show that $g^*$ induces 
\begin{equation}\label{gOO}
 \mathcal{O}^{\times} _{\mathcal{X}_{2,\overline{K}}} /
 (\mathcal{O}^{\times} _{\mathcal{X}_{2,\overline{K}}})^{\ell ^N} 
 \to 
 \mathcal{O}^{\times} _{\mathcal{X}_{1,\overline{K}}} /
 (\mathcal{O}^{\times} _{\mathcal{X}_{1,\overline{K}}})^{\ell ^N} ;\ 
 S_2 \mapsto S_1 ^m . 
\end{equation} 
We consider the natural isomorphism 
$V_j \cong A_K(|c_j|,1)$ 
given by the parameter $S_j$. 
For $0< t <1$, we define a circle 
$C_t$ over $\mathbf{C}$ by 
\[
 C_t (\mathbf{C} )= 
 \{ x \in \mathbf{C} \mid |x|=t \}. 
\]
For $c \in \overline{K}$ such that 
$|c|<1$ and $|c|$ is sufficiently close to $1$, 
the morphism $g$ induces a finite flat morphism 
\[
 C_{|c|} \cong \Sp \mathbf{C} \langle X_1 ,X_1 ^{-1} \rangle 
 \longrightarrow  
 C_{m|c|} \cong \Sp \mathbf{C} \langle X_2 ,X_2 ^{-1} \rangle 
\]
of degree $m$ 
such that $g^* (X_2 )= c' X_1 ^m g' (X_1)$, 
where $X_1 =S_1 /c$, $X_2 =S_2 /c^m$, 
$c' \neq 0 \in K(c)$ and 
\[
 g' (X_1 )=1 + \sum_{k \neq 0} a_k X_1 ^k \in 
 K(c) \langle X_1 ,X_1 ^{-1} \rangle 
\]
for $a_k \in K(c)$ satisfying $|a_k | <1$. 
We take such $c \in \overline{K}$. 
Then $g' (X_1 )$ is $\ell$-divisible in 
$K(c) \langle X_1 ,X_1 ^{-1} \rangle$. 
This shows \eqref{gOO}. 
Hence, we have that $g^* (\gamma_2 )=m\gamma_1$ by 
\eqref{Ocoh}. 
Then we have $g_* (\gamma_1 )=\gamma_2$, 
because 
$g_* \circ g^* =m$ on 
$H^1 (V_{2,\mathbf{C}} ,\mathbb{Q}_{\ell} )$ by 
\cite[Theorem 5.4.1.(d)]{BeEt}. 
\end{proof}

\begin{thm}
Let $f \colon W_1 \to W_2$ be 
a finite flat morphism 
that is compatible with 
semi-stable coverings. 
Then $f$ induces the following 
commutative diagrams: 
\begin{align*}
 &\xymatrix{
 0 \ar@{->}[r]  & 
 H^1 (\mathscr{W}_{1,\overline{k}} , 
 \mathbb{Q}_{\ell} ) \ar@{->}[r] 
 \ar@{->}[d]_{\hat{f}_{\overline{k} *}} & 
 H^1 (W_{1,\mathbf{C}}, \mathbb{Q}_{\ell} ) \ar@{->}[r] 
 \ar@{->}[d]_{f_{\mathbf{C} *}} &
 H_1 (\widetilde{\Gamma}_{\mathscr{W}_1} , 
 \mathbb{Q}_{\ell}) (-1) \ar@{->}[r] 
 \ar@{->}[d]_{\widetilde{\phi}_{f *}} & 0 \\ 
 0 \ar@{->}[r]  & 
 H^1 (\mathscr{W}_{2,\overline{k}} , 
 \mathbb{Q}_{\ell} ) \ar@{->}[r] & 
 H^1 (W_{2,\mathbf{C}}, \mathbb{Q}_{\ell} ) \ar@{->}[r] &
 H_1 (\widetilde{\Gamma}_{\mathscr{W}_2} , 
 \mathbb{Q}_{\ell}) (-1) \ar@{->}[r] & 
 0, 
 }\\
 &\xymatrix{
 0 \ar@{->}[r]  & 
 H^1 (\Gamma_{\mathscr{W}_1} , 
 \mathbb{Q}_{\ell} ) \ar@{->}[r] 
 \ar@{->}[d]_{\phi_{f *}} & 
 H^1 (\mathscr{W}_{1,\overline{k}}, \mathbb{Q}_{\ell} ) \ar@{->}[r] 
 \ar@{->}[d]_{\hat{f}_{\overline{k} *}} &
 \bigoplus_{v \in \mathcal{V}
 (\Gamma_{\mathscr{W}_1})} 
 H^1 (\widetilde{Y}_v , 
 \mathbb{Q}_{\ell}) \ar@{->}[r] 
 \ar@{->}[d]_{\hat{f}_{\overline{k} *}} & 0 \\ 
 0 \ar@{->}[r]  & 
  H^1 (\Gamma_{\mathscr{W}_2} , 
 \mathbb{Q}_{\ell} ) \ar@{->}[r] & 
 H^1 (\mathscr{W}_{2,\overline{k}}, \mathbb{Q}_{\ell} ) 
 \ar@{->}[r] &
 \bigoplus_{v \in \mathcal{V}
 (\Gamma_{\mathscr{W}_2})} 
 H^1 (\widetilde{Y}_v , 
 \mathbb{Q}_{\ell}) \ar@{->}[r] & 
 0, 
 }
\end{align*}
\begin{align*}
 &\xymatrix{
 0 \ar@{->}[r]  & 
 H^1 (\mathscr{W}_{2,\overline{k}} , 
 \mathbb{Q}_{\ell} ) \ar@{->}[r] 
 \ar@{->}[d]_{\hat{f}^* _{\overline{k}}} & 
 H^1 (W_{2,\mathbf{C}}, \mathbb{Q}_{\ell} ) \ar@{->}[r] 
 \ar@{->}[d]_{f^* _{\mathbf{C}}} &
 H_1 (\widetilde{\Gamma}_{\mathscr{W}_2} , 
 \mathbb{Q}_{\ell}) (-1) \ar@{->}[r] 
 \ar@{->}[d]_{{\widetilde{\phi}^* _f}} & 0 \\ 
 0 \ar@{->}[r]  & 
 H^1 (\mathscr{W}_{1,\overline{k}} , 
 \mathbb{Q}_{\ell} ) \ar@{->}[r] & 
 H^1 (W_{1,\mathbf{C}}, \mathbb{Q}_{\ell} ) \ar@{->}[r] &
 H_1 (\widetilde{\Gamma}_{\mathscr{W}_1} , 
 \mathbb{Q}_{\ell}) (-1) \ar@{->}[r] & 
 0, 
 }\\
 &\xymatrix{
 0 \ar@{->}[r]  & 
 H^1 (\Gamma_{\mathscr{W}_2} , 
 \mathbb{Q}_{\ell} ) \ar@{->}[r] 
 \ar@{->}[d]_{\phi^* _f} & 
 H^1 (\mathscr{W}_{2,\overline{k}}, \mathbb{Q}_{\ell} ) \ar@{->}[r] 
 \ar@{->}[d]_{\hat{f}^* _{\overline{k}}} &
 \bigoplus_{v \in \mathcal{V}
 (\Gamma_{\mathscr{W}_2})} 
 H^1 (\widetilde{Y}_v , 
 \mathbb{Q}_{\ell}) \ar@{->}[r] 
 \ar@{->}[d]_{\hat{f}^* _{\overline{k}}} & 0 \\ 
 0 \ar@{->}[r]  & 
 H^1 (\Gamma_{\mathscr{W}_1} , 
 \mathbb{Q}_{\ell} ) \ar@{->}[r] & 
 H^1 (\mathscr{W}_{1,\overline{k}}, \mathbb{Q}_{\ell} ) 
 \ar@{->}[r] &
 \bigoplus_{v \in \mathcal{V}
 (\Gamma_{\mathscr{W}_1})} 
 H^1 (\widetilde{Y}_v , 
 \mathbb{Q}_{\ell}) \ar@{->}[r] & 
 0. 
 }
\end{align*}
\end{thm}
\begin{proof}
We can easily check the commutativities of 
the second and fourth diagrams 
from the construction of the 
short exact sequences. 
The commutativities of 
the former halves of 
the first and third diagrams 
are trivial. 

For $e \in \mathcal{E}(\Gamma_{\mathscr{W}'_j})$, 
we define 
\begin{align*}
 \gamma_e \in 
 &H^0 \Bigl( \mathscr{W}_{j,\overline{k}} , 
 \bigl( ( 
 {\pi_{t(e)}}_* \mathbb{Q}_{\ell} \oplus 
 {\pi_{s(e)}}_* \mathbb{Q}_{\ell}
 ) / \mathbb{Q}_{\ell} \bigr)^* (-1) \Bigr) \\ 
 &\subset 
 H^0 \biggl( \mathscr{W}_{j,\overline{k}} , 
 \biggl( \Bigl( \bigoplus_{v \in \mathcal{V}
 (\Gamma_{\mathscr{W}'_j})} 
 {\pi_v }_* \mathbb{Q}_{\ell} 
 \Bigr) \Big/ \mathbb{Q}_{\ell} \biggr)^* (-1)
 \biggr) 
\end{align*}
by 
$\gamma_e \bigl( (a,a' ) \bigr) =a -a'$ for 
$(a,a' ) \in H^0 \bigl( \mathscr{W}_{j,\overline{k}} , 
 ( {\pi_{t(e)}}_* \mathbb{Q}_{\ell} \oplus 
 {\pi_{s(e)}}_* \mathbb{Q}_{\ell}
 ) / \mathbb{Q}_{\ell} \bigr)$, 
and consider $\gamma_e$ as an element of 
$H^0 (\mathscr{W}_{j,\overline{k}} , 
 R^1 \varPsi_{\mathscr{W}_{j,\mathcal{O}_{\mathbf{C}}}} 
 \mathbb{Q}_{\ell} )$ 
by the canonical isomorphism 
\[
 R^1 \varPsi_{\mathscr{W}_{j,\mathcal{O}_{\mathbf{C}}}} 
 \mathbb{Q}_{\ell} \cong 
 \biggl( \Bigl( \bigoplus_{v \in \mathcal{V}
 (\Gamma_{\mathscr{W}'_j})} 
 {\pi_v }_* \mathbb{Q}_{\ell} 
 \Bigr) \Big/ \mathbb{Q}_{\ell} \biggr)^* (-1) 
\]
induced from \eqref{R^1} for $\mathscr{W}_j$. 
We define $\gamma_{e} =0$ 
for 
$e \in \mathcal{E}(\widetilde{\Gamma}_{\mathscr{W}_j})
 \setminus \mathcal{E}(\Gamma_{\mathscr{W}'_j})$. 

We show the commutativity of 
the latter half of 
the first diagram. 
We consider a cycle $R =e_1 \cdots e_m$ of 
$\widetilde{\Gamma}_{\mathscr{W}_1}$ 
as an element of 
$H_1 (\widetilde{\Gamma}_{\mathscr{W}_1} , 
 \mathbb{Q}_{\ell})$. 
Then it corresponds to 
$\sum_{i=1} ^m \gamma_{e_i} 
 \in H^0 ( \mathscr{W}_{1,\overline{k}}, 
 R^1 \varPsi_{\mathscr{W}_{1,\mathcal{O}_{\mathbf{C}}}}
 \mathbb{Q}_{\ell} )$. 
We have 
\[
 f_* \biggl( \sum_{i=1} ^m \gamma_{e_i} \biggr) = 
 \sum_{i=1} ^m \gamma_{\widetilde{\phi}_{f,E} (e_i )} 
 \in H^0 ( \mathscr{W}_{2,\overline{k}}, 
 R^1 \varPsi_{\mathscr{W}_{2,\mathcal{O}_{\mathbf{C}}}}
 \mathbb{Q}_{\ell} ) 
\] 
by Lemma \ref{loc}, 
and this corresponds to 
$\widetilde{\phi}_{f *} (R)$. 

We show the commutativity of 
the latter half of 
the third diagram. 
We consider a cycle $R' =e' _1 \cdots e' _m$ of 
$\widetilde{\Gamma}_{\mathscr{W}_2}$ 
as an element of 
$H_1 (\widetilde{\Gamma}_{\mathscr{W}_2} , 
 \mathbb{Q}_{\ell})$. 
Then it corresponds to 
$\sum_{i=1} ^m \gamma_{e'_i} 
 \in H^0 ( \mathscr{W}_{2,\overline{k}}, 
 R^1 \varPsi_{\mathscr{W}_{2,\mathcal{O}_{\mathbf{C}}}}
 \mathbb{Q}_{\ell} )$. 
We have 
\[
 f^* \biggl( \sum_{i=1} ^m \gamma_{e' _i} \biggr) = 
 \sum_{i=1} ^m 
 \sum_{e \in \widetilde{\phi}_{f,E} ^{-1} (e'_i) }
 n_e \gamma_e 
 \in H^0 ( \mathscr{W}_{1,\overline{k}}, 
 R^1 \varPsi_{\mathscr{W}_{1,\mathcal{O}_{\mathbf{C}}}}
 \mathbb{Q}_{\ell} ) 
\] 
by Lemma \ref{loc}, 
and this corresponds to 
$\widetilde{\phi}_f ^* (R' )$. 
\end{proof}

\noindent
Naoki Imai\\ 
Graduate School of Mathematical Sciences, 
The University of Tokyo, 3-8-1 Komaba, Meguro-ku, 
Tokyo, 153-8914, Japan\\ 
naoki@ms.u-tokyo.ac.jp\\ 

\noindent
Takahiro Tsushima\\ 
Graduate School of Mathematical Sciences, 
The University of Tokyo, 3-8-1 Komaba, Meguro-ku, 
Tokyo, 153-8914, Japan\\ 
tsushima@ms.u-tokyo.ac.jp\\

\end{document}